\documentclass[11pt,a4paper,reqno]{amsart}

\usepackage{pdfsync}

\usepackage{mathdots}
\usepackage{tikz}
\usetikzlibrary{decorations.markings}
\usetikzlibrary{decorations.pathreplacing,shapes,arrows}

\usetikzlibrary{arrows,matrix}
\usetikzlibrary{snakes}
\usepgflibrary{arrows}
\tikzset{
    >=stealth',
    punkt/.style={
           rectangle,
           rounded corners,
           draw=black, very thick,
           text width=6.5em,
           minimum height=2em,
           text centered},
    pil/.style={
           ->,
           thick,
           shorten <=2pt,
           shorten >=2pt,}
}
\tikzset{every loop/.style={min distance=2mm,in=225,out=135,looseness=10}}
\tikzset{->-/.style={decoration={
  markings,
  mark=at position #1 with {\arrow{>}}},postaction={decorate}}}
\newcommand{\redarrowdraw}[3]{\draw[->-=#1, thick, color=red] (#2) to (#3);}
\newcommand{\arrowdraw}[3]{\draw[->-=#1, thick] (#2) to (#3);}
\newcommand{\smallredarrowdraw}[3]{\draw[->-=#1, color=red] (#2) to (#3);}
\newcommand{\smallarrowdraw}[3]{\draw[->-=#1] (#2) to (#3);}

\usepackage{tikz}
\usetikzlibrary{decorations.markings}
\usetikzlibrary{arrows,matrix}
\usepgflibrary{arrows}

\newcommand\GreenL{\mathcal{L}}
\newcommand\GreenR{\mathcal{R}}
\newcommand\GreenH{\mathcal{H}}  
\newcommand\GreenD{\mathcal{D}}

\usepackage{clrscode}
\usepackage{mathrsfs,amssymb}

\usepackage{amsmath, amsthm}

\newtheorem{theorem}{Theorem}[section]

\newtheorem{question}[theorem]{Question}
\newtheorem{lemma}[theorem]{Lemma}
\newtheorem{corollary}[theorem]{Corollary}

\newtheorem{proposition}[theorem]{Proposition}

\theoremstyle{remark}

\newcommand{\sgo}{S\Gamma(1)}

\newcommand{\gh}{\mathcal{H}}

\begin{document}

\title[Maximal Subgroups of Special Inverse Monoids]{Maximal Subgroups of Finitely Presented Special Inverse Monoids}

\maketitle

\begin{center}
    ROBERT D. GRAY\footnote{School of Mathematics, University of East Anglia, Norwich NR4 7TJ, England.
Email \texttt{Robert.D.Gray@uea.ac.uk}.}
\ and \ MARK KAMBITES\footnote{Department of Mathematics, University of Manchester, Manchester M13 9PL, England. Email \texttt{Mark.Kambites@manchester.ac.uk}.

This research was supported by the 
EPSRC-funded projects 
EP/N033353/1 `Special inverse monoids: subgroups, structure, geometry, rewriting systems and the word problem' 
and 
EP/V032003/1 `Algorithmic, topological and geometric aspects of infinite groups, monoids and inverse semigroups', and by a London Mathematical Society Research Reboot Grant.
} 
\end{center}

\begin{abstract}
We study the maximal subgroups (also known as group $\GreenH$-classes) of finitely presented special inverse monoids. We show that the maximal subgroups which can arise in such monoids are exactly the recursively presented groups, and moreover every such maximal subgroup can also arise in the $E$-unitary case. We also prove that the possible groups of units are exactly the finitely generated recursively presented groups; this improves upon a result of, and answers a question of, the first author and Ru\v{s}kuc. These results give the first significant insight into the maximal subgroups of such monoids beyond the group of units, and the results together demonstrate that (perhaps surprisingly) it is possible for the subgroup structure to have a complexity which significantly exceeds that of the group of units.  We also observe that a finitely presented special inverse monoid (even an $E$-unitary one) may have infinitely many pairwise non-isomorphic maximal subgroups; this contrasts sharply with the case of (non-inverse) special monoids, where Malheiro showed that all idempotents lie in the $\GreenD$-class of $1$, from which it follows that all maximal subgroups are isomorphic.
\end{abstract}

\section{Introduction}\label{sec_intro}
This paper is concerned with inverse monoids that admit presentations in which each defining relation has the form $w = 1$. The study of these monoids, which are termed  \emph{special inverse monoids} in the literature, is motivated both intrinsically by a beautiful geometric theory, and extrinsically by connections to other areas of semigroup theory and geometric group theory, including most notably possible applications to the \textit{one-relator word problem} (the question of whether  the word problem for one-relator monoids is decidable). The one-relator word problem has been a natural open problem since the resolution of the corresponding problem for groups by Magnus in 1932; it has resisted extensive study (most notably by Adjan and co-authors) and is widely regarded as one of the hardest and most important open problems in semigroup theory.\footnote{
For the avoidance of confusion, we note that neither one-relator groups nor one-relator (special) inverse monoids are typically examples of one-relator monoids, since interpreting
a group or inverse monoid presentation as a monoid presentation will usually give a different monoid. Indeed, non-trivial free groups require relations to define them as monoids (in fact it can always be done with just a single relation \cite{Perrin1984}), while non-trivial free inverse monoids are not even finitely presented as monoids \cite{Schein1975}.  A notable exception is the \textit{bicyclic monoid}, which is given by the presentation $\langle p, q \mid pq = 1 \rangle$ as both a monoid and an
inverse monoid. In contrast, it transpires that one-relator groups \textit{are} examples of one-relator special inverse monoids. (This can be deduced from, for example,
 Lemma~\ref{lemma_removeidempotents} below).}

The study of special inverse monoids was initiated by Margolis, Meakin and Stephen \cite{MMS87}, motivated by successes in understanding
special (non-inverse) monoids, and the prospect of applying exciting geometric methods (such as the Scheiblich/Munn description of \textit{free inverse monoids} and
\textit{Stephen's folding procedure} - see below and the survey \cite{Meakin:2007zt} for details) which had recently been developed for the study of inverse monoids. They 
showed that Stephen's procedure specialises to give a particularly beautiful geometric theory of Sch\"utzenberger graphs in the case of special inverse monoids, and
could be used to prove analogues of some of the results about the special (non-inverse) case. The case of special \textit{one-relator} inverse monoids $\mathrm{Inv}\langle A \mid w=1 \rangle$ received particular attention, culminating in a celebrated proof of Ivanov, Margolis and Meakin \cite{Ivanov:2001kl} that the one-relator word problem (for general monoids) reduces to the word problem for certain one-relator \textit{special inverse} monoids. Since special inverse monoids have much more evident geometric structure than general monoids, this opened up the possibility that the kind of methods employed in geometric group theory could be used to understand the word problem for one-relator special inverse monoids, and
hence the original one-relator word problem for monoids. This motivated extensive further study of this area, with initial results including positive solutions to the word problem in certain important cases \cite{Hermiller:2010bs, Margolis:2005il, Meakin:2007zt}, but the first author \cite{Gray2020} eventually showing that, surprisingly, the word problem for special one-relator inverse monoids in general is undecidable. However, the case shown undecidable is not one arising from the reduction
of the one-relator word problem for monoids, so the latter remains open and a better understanding of special inverse monoids may yet produce a solution.

In order to progress further, it seems that a deeper and more systematic understanding of special inverse monoids is required. In particular, it is necessary to
understand the subgroup structure of such monoids. 
Associated with every idempotent $e$ of a monoid is a group $H_e$ called the \textit{group $\GreenH$-class} of $e$, which is the largest subgroup of the monoid (with respect to containment) containing $e$. For this reason, these are also called the \emph{maximal subgroups} of the monoid, and every subgroup of the monoid is contained in a group $\GreenH$-class. The group $\GreenH$-class of the identity element is just the group of units of the monoid and the group $\GreenH$-classes of distinct idempotents are disjoint. Hence the main task in investigating the subgroup structure of a monoid is to understand its group $\GreenH$-classes.    

An obvious question is to what extent the subgroup structure resembles that found in the well-established theory of special (non-inverse) monoids, where it is known that all group $\GreenH$-classes are isomorphic to the group of units \cite{Malheiro2005} and that one-relator examples have group of units (and
therefore group $\GreenH$-classes) which are one-relator groups. (This fact has fact played a key role in Adjan's resolution of some cases of the one-relator word problem \cite{Adjan:1966bh}.) The first author and Ru\v{s}kuc \cite{GrayRuskucUnits} recently investigated 
the (left, right and two-sided) units, showing that in general the behaviour is very different from the non-inverse case: unlike in the non-inverse case, there is a one-relator special inverse monoid whose group of units is not a one-relator group, and there is a finitely presented special inverse monoid whose group of units is not finitely presented. 
One aim of this paper, realised in Section~\ref{sec_GroupsOfUnits} below, is to answer a question they posed \cite[Question 8.6]{GrayRuskucUnits} by giving a complete characterisation of the possible groups of units in finitely presented special inverse monoids: it transpires that they are exactly the finitely generated, recursively presented groups (Theorem~\ref{thm_possibleunits}).

The other main aim of this paper, in Section~\ref{sec_GroupHClasses}, is to initiate the study of group $\GreenH$-classes more generally. In sharp contrast to the case of special (non-inverse) monoids, and contrary to our own prior expectations and we believe to those of most experts in the field, we show that these can differ wildly from the group of units. We
introduce a powerful construction (Theorem~\ref{thm_stabiliser}) which allows us to exactly characterise the possible group $\GreenH$-classes in finitely presented
special inverse monoids: these are exactly the (not necessarily finitely generated) recursively presented groups (Corollary~\ref{cor_possiblemaxsubgroups}). The
same construction, combined with an old result of Higman \cite[Theorem~7.3]{Lyndon:2001lh}, also allows us to produce an example of a single finitely presented special inverse monoid in which every finite group arises as an $\GreenH$-class (Corollary~\ref{cor_everyfinitegroup}).

The main theme of this paper is that the subgroups of special inverse monoids are potentially far wilder, and the structure of these monoids therefore more complex, than expected, but this does not mean there is no hope of understanding them. Indeed, in order to establish these wild examples we develop a new geometric approach to maximal subgroups, exploiting the fact  \cite[Theorem~3.5]{Stephen:1990ss} that they are isomorphic to the automorphism groups of the \textit{Sch\"utzenberger graphs} of the monoid. This approach, which contrasts with the more algebraic/combinatorial approach of \cite{Gray2020} and \cite{GrayRuskucUnits}, also offers new ways to obtain a positive understanding of group $\GreenH$-classes in particular cases and important sub-classes, which we will develop in future work.

\section{Preliminaries}\label{sec_preliminaries}
In this section we fix notation and briefly recall some relevant background material from the geometric and combinatorial study of inverse monoids. 
For additional background we refer the reader to  
\cite{Meakin:2007zt} for combinatorial inverse semigroup theory and  
\cite{Lyndon:2001lh} for combinatorial group theory.

\subsection*{Graphs}
Throughout this paper we use the word \emph{graph} to mean a (possibly infinite) directed graph, possibly with loops and multiple edges, in which each edge is labelled by a symbol from some alphabet. The graph is called \emph{bi-deterministic} if no two edges with the same label share a start vertex or an end vertex.
 A \emph{path} in a graph is a sequence of edges $e_1, \ldots, e_n$ such that the end vertex of $e_{i}$ coincides with the start vertex of $e_{i+1}$ for $1 \leq i \leq n-1$. The
path is called \emph{closed} if its start vertex (the start vertex of $e_1$) coincides with its end vertex (the end vertex of $e_n$). The \emph{label} of a path is the word which is the concatenation in order of the labels of the edges. A path is called \emph{simple} if 
no two edges start at the same vertex or end at the same vertex.
A graph $\Gamma$ is a \emph{subgraph} of a graph $\Omega$ if the vertex and edge sets of $\Gamma$ are subsets of the vertex and edge sets respectively of $\Omega$; it is said to be an \emph{induced subgraph} if
in addition it contains every edge of $\Omega$ whose start and end are vertices of $\Gamma$.
A \emph{morphism} $\phi: \Gamma \rightarrow \Delta$ of graphs consists of a map from the vertex set of $\Gamma$ to the vertex set of $\Delta$ and a map from
the edge set of $\Gamma$ to the edge set of $\Delta$ which for each edge preserves the label and respects the start and end vertices.

\subsection*{Inverse monoids and Sch\"utzenberger graphs}
An \emph{inverse monoid} $M$ is a monoid such that for every $m \in M$ there is a unique element $m^{-1} \in M$, called the \emph{inverse} of $m$, that satisfies $mm^{-1}m=m$ and $m^{-1} m m^{-1} = m^{-1}$.  
The map $m \mapsto m^{-1}$ satisfies $(m^{-1})^{-1} = m$ and $(mn)^{-1} = n^{-1} m^{-1}$. If $A$ is a subset of $M$ we
write $A^{-1}$ for the set of inverses of elements in $A$.
For brevity, we shall also sometimes use the notation $m'$ for the inverse of $m$, especially when working with inverse monoid presentations.
We say an inverse monoid is \emph{generated} by a subset $A$ if is generated by $A$ under the multiplication and inversion operations, or equivalently, generated
by $A \cup A^{-1}$ under multiplication alone.
If $M$ is an inverse monoid generated by a set $A$ then the \emph{Cayley graph} of $M$ with respect to $A$ has vertex set $M$ and a directed edge from $m$ to $mx$ labelled by $x$ for all $m \in M$ and $x \in A \cup A^{-1}$. 
The \emph{Sch\"utzenberger graphs} of $M$ with respect to the generating set $A$ are the strongly connected components of the Cayley graph where two vertices $u, v$ belong to the same strongly connected component if there is a path from $u$ to $v$, and a path from $v$ back to $u$. 
So the Sch\"utzenberger graph $S\Gamma(m)$ of $m \in M$ is the subgraph of the Cayley graph induced on the set of vertices in the strongly connected component containing $m$.  Note that if $M$ is a group then it has just one Sch\"utzenberger graph, which is the entire Cayley graph.

Recall that in an inverse monoid $M$ two elements $m$ and $n$ are $\GreenR$-related if and only if they generate the same principal right ideal which is equivalent to saying that $mm^{-1} = nn^{-1}$.
Dually $m$ and $n$ are $\GreenL$-related if $m^{-1}m = n^{-1}n$ and are $\GreenH$-related if they are both $\GreenR$- and $\GreenL$-related.   
If $m \in M$ we write $S \Gamma(m)$ for the Sch\"utzenberger graph of $m$, and $H_m$, $R_m$ etc for the equivalence classes
of $M$ under Green's relations. Note that the vertex set of $S \Gamma(m)$ is exactly $R_m$. We define the \textit{root} of the graph $S \Gamma(m)$ to be the vertex $m m^{-1}$, that is, the unique vertex corresponding to an idempotent element of $M$.
 If $w \in (A^{\pm 1})^*$ is a word over the generating set and its inverses viewed as a formal alphabet then we may use $w$ in place of $m$ in the above notation, to be interpreted as the element
of $M$ represented by $w$.
It may be shown that if $s \GreenR t$ and $su=t$ then $tu^{-1}=s$.      
It follows that if $mx \GreenR x$ where $m \in M$ and $x \in A^{\pm 1}$ then $mxx^{-1} =m$ and hence within the Sch\"utzenberger graphs edges come in inverse pairs. 
For simplicity we will often draw and speak of Sch\"utzenberger graphs of an $A$-generated inverse monoid as only having edges labelled by letters from $A$, leaving it implicit that there are always edges labelled by their inverses in the reverse direction; we shall sometimes speak of traversing edges ``backwards'', by which we formally mean traversing the corresponding inverse edge.

\subsection*{Inverse monoid presentations and the maximal group image}
Let $A$ be a (not necessarily finite) alphabet. The \textit{free inverse monoid} $\mathrm{Inv} \langle  A \rangle$ is the unique (up to isomorphism) inverse monoid generated by $A$ with the property that every map from $A$ to an inverse monoid extends to a morphism from $\mathrm{Inv} \langle  A \rangle$. More concretely, if we let $A^{-1}$ be a
set of formal inverses for the generators in $A$ and write $A^{\pm 1}$ for $A \cup A^{-1}$, it is the quotient of the free monoid $(A^{\pm 1})^*$ by the \textit{Wagner congruence}, which is the congruence generated by the relations $ww^{-1}w = w$ and $uu^{-1}ww^{-1} = ww^{-1}uu^{-1}$ where $u, w \in (A^{\pm 1})^*$ and where by definition $(a_1^{\epsilon_1} \ldots a_k^{\epsilon_k})^{-1} = a_k^{-\epsilon_k} \ldots a_1^{-\epsilon_1}$. 
 An elegant geometric description for the free inverse monoid was given by Munn \cite{Munn74}, based on earlier work
of Scheiblich \cite{Scheiblich73}; we shall not make explicit use of this, but the ideas they developed are central to geometric inverse semigroup theory and implicitly
used in much of what we do. A word $w$ over $A^{\pm 1}$ is called a \textit{fundamental idempotent} if it represents an idempotent element of the free inverse monoid on $A$. This is equivalent
to saying that $w$ represents an idempotent element in every $A$-generated inverse monoid, and also to saying that $w$ represents the identity in the free group on $A$.

The inverse monoid defined by the presentation  
$\mathrm{Inv}\langle A \mid R \rangle$, 
where $R \subseteq (A^{\pm 1})^* \times (A^{\pm 1})^* $
 is the quotient of the free inverse monoid $\mathrm{Inv} \langle  A \rangle$ by the congruence generated by 
$R$.
An inverse monoid presentation is called \textit{special} if all relations have the form $w=1$, and an inverse monoid is called special if it admits a
special inverse monoid presentation. A well-known fact about special inverse monoid presentations is the following result, which essentially says that fundamental idempotent relators can be incorporated into other relators.
\begin{lemma}[see for example \cite{Gray2020}, Lemma 3.3]\label{lemma_removeidempotents}
Let $A$ be an alphabet and $e, r_1, \dots, r_m \in (A^{\pm 1})^*$ with $e$ a fundamental idempotent. Then
$$\mathrm{Inv}\langle A \mid e = r_1 = r_2 \dots = r_m = 1 \rangle \ = \ \mathrm{Inv} \langle A \mid e r_1 = r_2 = \dots = r_m = 1 \rangle.$$
\end{lemma}

We use $\mathrm{Gp}\langle A \mid R \rangle$ to denote the group defined by the presentation with generators $A$ and defining relators $R$, and we use $\mathrm{Gp}\langle X \rangle$ to denote the free group on the set $X$. 
Let $M = \mathrm{Inv}\langle A \mid R \rangle$ and let $G = \mathrm{Gp}\langle A \mid R \rangle$ be its \emph{maximal group image} and let $\sigma:M \rightarrow G$ denote the canonical surjective homomorphism from $M$ to $G$. Let $\Gamma$ be the Cayley graph of $G$ with respect to $A$. This defines a map $\overline{\sigma}$ from the disjoint union of the Sch\"utzenberger graphs of $M$ to the Cayley graph $\Gamma$ of $G$, which maps vertices using $\sigma$ and each edge $m \xrightarrow{a} ma$ in a      
Sch\"utzenberger graph maps to the unique edge in $\Gamma$ from $\sigma(m)$ to $\sigma(ma)$ labelled by $a$.     
This map $\overline{\sigma}$ is clearly a morphism of labelled graphs.  
The inverse monoid $M$ is called \textit{$E$-unitary} if $\sigma^{-1}(1_G)$ is equal to the set of idempotents of $M$. It is known (see \cite[Lemma~1.8]{Stephen93}) that this is equivalent to the map $\overline{\sigma}$ being injective on every Sch\"utzenberger graph of $M$.

\subsection*{Stephen's procedure}

We now describe a procedure due to Stephen \cite{Stephen:1990ss} for iteratively approximating Sch\"utzenberger graphs of inverse monoids. 
This procedure is the graphical version of the Todd--Coxeter coset enumeration procedure from group theory generalised to inverse semigroup theory.
Since we will be working only with special inverse monoids in this paper we will only describe the procedure in this case, even though it applies more generally.

Let $M = \mathrm{Inv}\langle A \mid R \rangle$ be a special inverse monoid presentation and let $\Gamma$ be a graph labelled over $A \cup A^{-1}$ such that the edges in $\Gamma$ occur in inverse pairs.     
We define two operations:
\begin{enumerate}
\item $P$-expansion: For every vertex $v \in V(\Gamma)$ and every $r \in R$, if there is no closed path at $v$ labelled by $R$ then we attach a simple closed path at $v$ labelled by $r$ such that all the internal vertices of this closed path are disjoint from $\Gamma$. (For every new edge we add we also add the corresponding inverse edge so that all the edges still occur in inverse pairs.)            
\item Edge folding: if there are edges $e$ and $f$ with the same label and the same start or end vertex then we identify these edges (which also identifies their start vertices and identifies their end vertices).   
  \end{enumerate}
It follows from \cite{Stephen:1990ss} that starting with any such graph $\Gamma$ and special inverse monoid presentation, the set of all graphs obtained by applying successive $P$-expansions and edge foldings forms a directed system in the category of $(A \cup A^{-1})$-labelled graphs. We denote the limit of this system by $\mathrm{Exp}(\Gamma)$. 
Given any word $w \in (A^{\pm 1})^*$ we use $L_w$ to denote the straight line graph labelled by the word $w$. 
So $L_w$ has vertex set $\{0, 1, \ldots, |w|\}$, one pair of inverse edges between $i$ and $i+1$ for all $0 \leq i \leq |w|-1$, and the label of the unique path of length $|w|$ from the vertex $0$ to the vertex $|w|$ is equal to the word $w$.         
The following is then a theorem of Stephen \cite{Stephen:1990ss} specialised to the case of special inverse monoids.

\begin{theorem}\label{thm_StephenThm}
Let $M = \mathrm{Inv}\langle A \mid R \rangle$ be a special inverse monoid and let 
$w \in (A^{\pm 1})^*$. Then the Sch\"utzenberger graph $S\Gamma(w)$ is isomorphic to the graph $\mathrm{Exp}(L_w)$ obtained by Stephen's procedure starting with the straight line graph $L_w$ labelled by $w$.      
  \end{theorem}

\section{Geometry of Sch\"utzenberger Graphs}\label{sec_geometry}

In this section we establish some foundational results about the geometry of Sch\"utzenberger graphs in special inverse monoids, which will be used to establish
our main theorems below, and are also likely to be of wider use in the future development of the subject.

\subsection*{Extending automorphisms} 

We shall need the following technical result, which gives a sufficient condition for an automorphism of a subgraph of a Sch\"utzenberger graph to
extend to an automorphism of the containing graph.

\begin{lemma}\label{lem_StephenAutInvariant}
Let $M = \mathrm{Inv}\langle A \mid R \rangle$ be a special inverse monoid, 
let $e$ be a fundamental idempotent in $(A^{\pm 1})^*$, and let 
$\Omega$ be any connected subgraph of $S\Gamma(e)$
such that $\Omega$ contains the vertex $e$ of $S\Gamma(e)$ 
and the word $e$ can be read from the vertex $e$ entirely within $\Omega$. 
Then every automorphism of $\Omega$ extends uniquely to an automorphism of $S\Gamma(e)$.   
\end{lemma}
\begin{proof}
Consider an automorphism $\theta$ of $\Omega$. Suppose that it maps the vertex $e$ to $ew$ where $w$ is a word labelling a path in $\Omega$ from $e$ to its image vertex under $\theta$. 
Such a path exists since $\Omega$ is connected.   
Now since the word $e$ can be read from the vertex $e$ of $\Omega$ and there is an automorphism of $\Omega$ sending $e$ to $ew$ it follows that $e$ can also be read from $ew$. Hence, it follows that $ewe$ is $\GreenR$-related to $e$, and the path labelled by the word $ewe$ can be read within $\Omega$ starting at the vertex $e$.           
The inverse of the automorphism $\theta$ maps $ew$ to $e$ and must map $e$ to $ew^{-1}$. 
Indeed, since there is a path in $\Omega$ labelled by $w$ from $e$ to $ew$ it follows that there is a path labelled by $w$ from $\theta^{-1}(e)$ to $\theta^{-1}(ew)=e$ and hence $\theta^{-1}(e) = ew^{-1}$. 
Since there is an automorphism of $\Omega$ sending $e$ to $ew^{-1}$, and since the word $e$ can be read within $\Omega$ at $e$, it follows that the word $e$ can be read within $\Omega$ starting at $ew^{-1}$. In particular $ew^{-1}e$ is $\GreenR$-related to $e$. 
But if $ew^{-1}e$ is $\GreenR$-related to $e$ then its inverse $ewe$ is $\GreenL$-related to $e$, and we conclude that $ewe$ is $\GreenH$-related to $e$. 
Also, since $e$ labels a closed path in the graph $\Omega$ starting and ending at the vertex $e$, and $\theta$ is an automorphism of $\Omega$ sending $e$ to $ew$ it follows that $e$ also labels a closed path in $\Omega$ starting and ending at the vertex $ew$. Hence $ew = ewe$ is $\GreenH$-related to $e$, in other
words it belongs to the group $\GreenH$-class $H_e$. 
Now for any vertex in $\Omega$ since $\Omega$ is connected we can choose a word $u$ over $A \cup A^{-1}$ labelling a path from the vertex $e$ of $\Omega$ to that vertex. Since $\theta$ maps $e$ to $ew$ it follows that the vertex at the end of the path from $e$ labelled by $u$ must map to $ewu$. So the automorphism $\theta$ is given by the map $eu \mapsto ewu$. But $ew=ewe$ so this is the map $eu \mapsto (ew)eu$ where $ew \in H_e$. It is then a consequence of Green's Lemma \cite[Lemma~2.2.2]{Howie95} that left multiplication by $ew$ defines a bijection from $R_e$ to itself which induces an automorphism of the Sch\"utzenberger graph $S\Gamma(e)$. 
Hence the automorphism $\theta$ of $\Omega$ extends uniquely to an automorphism of $S\Gamma(e)$. Since $\theta$ was an arbitrary automorphism of $\Omega$ this completes the proof.      
\end{proof}
An alternative way of seeing why the previous lemma is true is via Stephen's procedure: 
since the word $e$ can be read inside $\Omega$ from the vertex $e$ of $S\Gamma(e)$, and it is a connected subgraph of $S\Gamma(e)$, it follows that the Sch\"utzenberger graph $S\Gamma(e)$ can be constructed by starting from $\Omega$ and using Stephen's procedure. 
Since Stephen's procedure is automorphism-invariant, it follows that every automorphism of $\Omega$ will extend (uniquely, since an automorphism of an inverse automaton is uniquely determined by what it is does to any single vertex) to an automorphism of $S\Gamma(e)$.

\subsection*{$E$-unitary special inverse monoids}

The following result, part of which is due directly to Stephen \cite{Stephen93} and part of which we deduce from his work, shows that the Sch\"utzenberger graphs of an $E$-unitary special inverse monoid are isomorphic to certain induced subgraphs of the Cayley graph of the maximal group image.   

\begin{lemma}\label{lem_fullSubgraph} 
Let $M = \mathrm{Inv}\langle A \mid R \rangle$ be an $E$-unitary special inverse monoid and let $w \in (A^{\pm 1})^*$, and let $\Gamma$ be the Cayley graph of the maximal group image  
 $G = \mathrm{Gp}\langle A \mid R \rangle$ with respect to $A$. 
Then the Sch\"utzenberger graph $S\Gamma(w)$ is embedded into the Cayley graph $\Gamma$ by $\overline{\sigma}$ as an induced subgraph. 
Moreover, the embedded copy of $S\Gamma(w)$ is the smallest subgraph of $\Gamma$ such that the word $w$ can be read from $1$ and every relator word
$r \in R$ can be read at every vertex of $S\Gamma(w)$.
\end{lemma}
\begin{proof}
The fact the graph embeds 
as an induced subgraph
is \cite[Lemma~3.5]{Stephen93}. For the final part, let $\Delta$ be the smallest subgraph of $\Gamma$ in which word $w$ can be read from $1$ and every relator word $r \in R$ can be read at every vertex of $\Delta$.    
It follows from Stephen's procedure (Theorem~\ref{thm_StephenThm}) that, as a subgraph of $\Gamma$, the graph $S\Gamma(w)$ contains all the vertices and edges of $\Delta$.    
Conversely, it follows from the argument in \cite{Stephen93} preceding the statement of \cite[Lemma~3.1]{Stephen93} that for every edge $x$ in the Sch\"utzenberger graph $S\Gamma(w)$ there is a path in $S\Gamma(w)$ from the vertex $1$ labelled by a word of the form $v p_1 p_2 \ldots p_k$ where $v$ is a prefix of $w$ and each $p_i$ is a prefix of some defining relator $r \in R$, such that this path traverses the edge $x$ in some direction. It follows that $S\Gamma(w)$ is contained in  $\Delta$ and hence $S\Gamma(w) = \Delta$.  
\end{proof}

\subsection*{Morphisms between labelled digraphs and Sch\"utzenberger graphs} 
In several arguments we will construct a labelled digraph and then assert the existence of a morphism to this graph from a Sch\"utzenberger graph. The key lemma we need for this is the following. 

\begin{lemma}\label{lem_graphmorphisms}
Let $M = \mathrm{Inv}\langle A \mid R \rangle$ be a special inverse monoid, and let $\Omega$ be a bi-deterministic $A$-labelled graph in which for every vertex $v$ in $\Omega$ and every $r \in R$ there is a closed path in $\Omega$ at $v$ labelled by $r$. Then for every vertex $w$ of $\Omega$ there is a morphism from $S\Gamma(1)$ to $\Omega$ that sends the root of $S\Gamma(1)$ to $w$.    
  \end{lemma}
\begin{proof} 
We need to define a map and then show that it is a well-defined morphism. 

Let $T$ be the infinite graph constructed iteratively from a single vertex by adding a loop labelled by $r$ at every vertex for every $r \in R$, but not performing any edge folding. 
We view each of these loops as oriented in such a way that the word $r$ is the label of the path given by reading the loop clockwise.    
By a \emph{proper subpath of a loop of $T$} we mean a path $\pi$ with initial vertex being the vertex at which the loop was attached in the construction of $T$, and $\pi$ is a simple path which traverses the loop clockwise but does not visit every vertex of the loop i.e. the end vertex of $\pi$ is not equal to the start vertex of $\pi$. Note that if $r \in R$ is the label of a loop in $T$ then any proper subpath of this loop is labelled by a proper prefix of the word $r$.         
From the construction it follows that every vertex $u$ of $T$ there is a unique sequence $(\pi_1, \pi_2, \ldots, \pi_k)$ where each $\pi_i$ is a proper subpath of a loop and $\pi_1 \pi_2 \ldots \pi_k$ is a path from the root of $T$ to $u$.      
We define a map from the vertex set of $T$ to vertices in $\Omega$ where the vertex $u$ with corresponding sequence $(\pi_1, \pi_2, \ldots, \pi_k)$ of proper subpaths of loops maps to the vertex in $\Omega$ obtained by following the path labelled by $p_1 \ldots p_k$ starting at the vertex $w$ of $\Omega$, where $p_i$ is the label of the path $\pi_i$ for $1 \leq i \leq k$.   
This gives a well-defined (by uniqueness of the sequences of proper subpaths of loops) map from the vertices of $T$ to the vertices of $\Omega$.   
As a consequence of the assumptions that $\Omega$ is bi-deterministic and in $\Omega$ every relator from $R$ can be read from every vertex in $\Omega$, this map extends uniquely to a morphism of graphs which maps edges of $T$ to the edges of $\Omega$. 
Let us use $\phi$ to denote this graph morphism from $T$ to $\Omega$.  

It follows from Stephen's procedure that $S\Gamma(1)$ is obtained by bi-determinising $T$. 
We claim that $\phi$ induces a well-defined graph morphism from $S\Gamma(1)$ to $\Omega$.       
To see this note that two vertices $v$ and $u$ of $T$ are identified in $S\Gamma(1)$ if and only if there is a path in $T$ between these vertices labelled by a word that freely reduces to the empty word in the free group. Since $\phi$ is a morphism it follows that there is a path in $\Omega$ between $\phi(v)$ and $\phi(u)$ labelled by the same word that freely reduces to the empty word in the free group. Since the graph $\Omega$ is bi-deterministic it follows that $\phi(v) = \phi(u)$. Hence $\phi$ induces a well-defined map from the vertices of $S\Gamma(1)$ to the vertices of $\Omega$.          
Two edges $e$ and $f$ of $T$ are identified in $S\Gamma(1)$ if and only if they have the same label, say $a \in A$, and their start vertices $v$ and $u$ are identified in $S\Gamma(1)$. 
But we have already seen that this means that $\phi(v) = \phi(u)$ which, since $\Omega$ is bi-deterministic means that both $e$ and $f$ must be mapped to the unique edge in $\Omega$ with start vertex $\phi(v) = \phi(u)$ and labelled by $a$.                
This shows that $\phi$ induces a well defined map from the edges of $S\Gamma(1)$ to the edges of $\Omega$.    

It remains to verify that $\phi$ induces a morphism of graphs from $S\Gamma(1)$ to $\Omega$.    
Let $e$ be an edge in $S\Gamma(1)$. Choose an edge $f$ in $T$ such that $f$ is equal to $e$ when $T$ is bi-determinised, that is, $f$ is a member of the equivalence class of edges that represented $e$. Since $\phi$ is a morphism from $T$ to $\Omega$ it follows that the start vertex of $f$ in $T$ maps to the start vertex of $\phi(f)$ in $\Omega$, and the end vertex of $f$ in $T$ maps to the end vertex of $\phi(f)$ in $\Omega$. 
But by definition $\phi(e) = \phi(f)$ and $\phi$ maps the start vertex $e$ to the same place as the start vertex of $f$, and similarly for the end vertices. 
It follows that $\phi$ induces a morphism of graphs from $S\Gamma(1)$ to $\Omega$. 
\end{proof}
 
Note that previous lemma is not true if one drops the condition that $\Omega$ is bi-deterministic.   
The following two results will be useful.

\begin{lemma}\cite[Corollary~3.2]{Gray2020}\label{lem_reducerightinvert}
Let $M = \mathrm{Inv}\langle A \mid R \rangle$ be an inverse monoid. 
If $xaa^{-1}y$ is right invertible where $a \in A \cup A^{-1}$ and $x, y \in (A \cup A^{-1})^*$ then $xaa^{-1}y = xy$ in $M$.     
  \end{lemma}

\begin{lemma}\label{lemma_eunitarytransfer}
Suppose $M = \mathrm{Inv}\langle A \mid R \rangle$ is an $E$-unitary inverse monoid, and $T$ is a set of relators which hold in the maximal
group image $\mathrm{Gp}\langle A \mid R \rangle$. Then the inverse monoid $N = \mathrm{Inv}\langle A \mid R \cup T \rangle$ is $E$-unitary.
\end{lemma}
\begin{proof}
Suppose $s \in N$ is an element which maps to $1$ in the maximal group image.  Choose a word $w$ over $A^{\pm 1}$ representing $s$. Since the relations in $T$ already hold in the maximal group image of $M$, $w$ also represents $1$ in the maximal group image of $M$. Since $M$ is $E$-unitary this means $w$ represents an idempotent
in $M$, and hence also in $N$ which is a quotient of $M$.
\end{proof}

\section{Exact characterisation of groups of units}\label{sec_GroupsOfUnits}

It is known that the group of units of a finitely presented special inverse monoid is always finitely generated: this is implicit in the work of Ivanov, Margolis and Meakin \cite[Proposition 4.2]{Ivanov:2001kl}, and an explicit proof can be found as \cite[Theorem 1.3]{GrayRuskucUnits}. The first author and Ru\v{s}kuc \cite{GrayRuskucUnits} have recently shown that for every finitely generated subgroup $H$ of a finitely presented group $G$, there is a finitely presented special inverse monoid with group of units the free product $G * H$. Thus, every finitely generated subgroup of a finitely presented group (which by Higman's embedding theorem
\cite{Higman61} means every finitely generated recursively presented group), is a free factor of the group of units of a finitely presented special inverse monoid. They ask \cite[Question 8.6]{GrayRuskucUnits} the natural question of whether the ``free factor'' can be eliminated, in other words, whether every finitely
generated recursively presented group arises as the group of units of some finitely presented special inverse monoid. In this section we answer this question in
the positive, giving an exact characterisation of the possible groups of units of finitely presented special inverse monoids.

\begin{theorem}\label{thm_possibleunits}
The groups of units of finitely presented special inverse monoids are exactly the finitely generated, recursively presented groups (or equivalently, the
finitely generated subgroups of finitely presented groups).
\end{theorem}
\begin{proof}
The fact that groups of units in finitely presented special inverse monoids are finitely generated is \cite[Theorem 1.3]{GrayRuskucUnits} (and also implicit in \cite[Proposition 4.2]{Ivanov:2001kl}). The fact that they are recursively presented can easily be shown by using the special inverse monoid presentation to recursively enumerate all relations $w=1$ which hold in the special inverse monoid, and then using those where $w$ factorises as a product of chosen representatives of a generating set for
the group of units to enumerate relations which hold in the group. (Alternatively, it follows from Corollary~\ref{cor_possiblemaxsubgroups} below, which is proved
independently of Theorem~\ref{thm_possibleunits}.) 

For the converse, let $H$ be a finitely generated, recursively presented group. By the Higman embedding theorem \cite{Higman61} we may assume that $H$ is a (finitely generated, and hence recursively enumerable) subgroup of a finitely presented group $G$. Choose a finite special monoid presentation for $G$, say $G = \mathrm{Mon}\langle A \mid r_1, \dots, r_k \rangle$, supposing
without loss of generality that $H$ is generated as a monoid by some 
(possibly empty if $H$ is the trivial group)
subset $B$ of $A$, and that for each $a \in A$ there is a 
unique 
formal inverse $\overline{a} \in A$ (with
$\overline{\overline{a}} = a$) and $a \overline{a} = 1$ is among the defining relators for $G$. Let $p_0$, \dots, $p_k$, $z$ and $d$ be symbols not in $A$,
and consider the special inverse monoid
\begin{align*}
M = \mathrm{Inv}\langle \ A, \ p_0, \dots, p_k, \ z, \ d \ \mid \ &p_i a p_i' p_i a' p_i' = 1 \ \ \ &(a \in A, i = 0, \dots, k) \\
&p_i r_i d' p_i' = 1 &(i = 1, \dots, k) \\
&p_0 d p_0' = 1 \\
&zbz'zb'z' = 1 &(b \in B) \\
&z \left( \prod_{i=0}^k p_i' p_i \right) z' = 1 \ \rangle.
\end{align*}
noting that the order of the product in the final relator is unimportant because the factors are fundamental idempotents.

We note that the presentation does \textbf{not} automatically identify the formal inverse of a generator in $A$ with its inverse in the inverse monoid, so $a$, $\overline{a}$,
$a'$ and $\overline{a}'$ may be four distinct elements of $M$. However, recalling that the defining relators for $G$ include relators of the form $a \overline{a}$ and $\overline{a} a$, notice that at any vertex in a Sch\"utzenberger graph of $M$ with edges labelled by all the $p_i$s coming in, and for any generator $a \in A$, Stephen's procedure will 
as a consequence of the second and third families of relators in the presentation for $M$  
attach closed paths labelled $a\overline{a}d'$, $\overline{a}ad'$ and $d$ which bi-determinise to give closed paths labelled $a\overline{a}$ and $\overline{a}a$. 
Thus, for every vertex $u$ with edges labelled by all the $p_i$s coming in, and for any generator $a \in A$, there is an edge leaving $u$ labelled by $a$ to a vertex $v$ with a parallel reverse edge from $v$ back to $u$ labelled by $\overline{a}$.  

Note also that the group defined by this presentation, which is the maximal group image of $M$, is easily seen to be a free product of $G$ (generated by $A$) with a free group freely generated by $\lbrace p_0, \dots, p_k, z \rbrace$, with $d$ mapping to $1$.

Let $\Gamma$ be the Cayley graph of $G$ with respect to $A$. 
Note that from the assumptions on the presentation chosen for $G$ every edge in $\Gamma$ labelled by $a$ has a unique reverse edge labelled by $\overline{a}$. 

We call a vertex of $\sgo$ a \textit{$z$-vertex} if it has a $z$-edge coming in, and a \textit{$p$-vertex} if it has edges coming in labelled by $p_i$ for all $i = 0, \dots, k$. Notice that it follows from the final relation in the presentation that every $z$-vertex is a $p$-vertex. 

We claim that every $p$-vertex of $\sgo$ has an embedded copy of $\Gamma$ with its root at this vertex. Indeed, let $v$ be a $p$-vertex. It follows by the first type of relation in the presentation that $v$ has $a$-edges going out for all $a \in A$, and that these edges all lead to $p$-vertices. Moreover, by the above argument, $v$ has $a$-edges coming in for each $a \in A$, and each comes from the vertex to which the corresponding $\overline{a}$-edge going out leads. A simple inductive argument shows that all words over
$A^{\pm 1}$ can be read from $v$ staying always at $p$-vertices. Now for each $p$-vertex and each defining relator $r_i$ of $G$, there must be closed
paths at $v$ labelled $d$ and $r_id'$, which since $\sgo$ is a bi-deterministic graph means there must be a closed path at every $p$-vertex labelled by $r_i$. 

It then follows from Lemma~\ref{lem_graphmorphisms} (applied with $G$ playing the role of the monoid, $\Gamma$ as its Sch\"utzenberger graph,
and $\sgo$ as the $\Omega$ in the statement of the lemma) that for each $p$-vertex $v$, there is a morphism from $\Gamma$ to $\sgo$ taking $1$ to $v$. 
Indeed, by the assumptions on the monoid presentation defining $G$
having the relations $a \overline{a}=1$ for all $a \in A$ it follows that    
$G = \mathrm{Mon}\langle A \mid r_1, \dots, r_k \rangle = 
 \mathrm{Inv}\langle A \mid r_1, \dots, r_k \rangle = T$. 
Hence it follows that working with this inverse monoid presentation for $G$ then the Cayley graph of $\Gamma$ of $G$ is isomorphic to $S\Gamma_T(1)$ the Sch\"utzenberger graph of $1$ of $G$ viewed as the inverse monoid $T$. 
Now from the observations in the previous paragraph it follows that all the $p$-vertices 
in $S\Gamma_M(1)$ 
that can be reached from a fixed $p$-vertex by a word over $A \cup A^{-1}$ is also a $p$-vertex, and for every vertex in that set every relator labels a closed path.      
Also the graph on this collection of $p$-vertices is clearly bi-deterministic since it is a subgraph of $S\Gamma_M(1)$ which is bi-deterministic.   
Hence the conditions Lemma~\ref{lem_graphmorphisms} are satisfied which completes the proof of the claim that for each $p$-vertex $v$, there is a morphism from $\Gamma$ to $S\Gamma_M(1)$ taking $1$ to $v$.  
Moreover, by our observations above about the maximal group image being the  
free product of $G$ (generated by $A$) with a free group freely generated by $\lbrace p_0, \dots, p_k, z \rbrace$, with $d$ mapping to $1$, it follows that 
if we compose this morphism from $\Gamma$ to $\sgo$ with the map from $\sgo$ to the maximal group image we see that distinct vertices of $\Gamma$ will map to elements which differ in the maximal group image, and therefore must be different in $M$, so that the map from $\Gamma$ to $\sgo$ is injective.
Note also that it follows from this argument that 
in addition to the edges labelled by generators from $A$  
the embedded copy of $\Gamma$ at a $p$-vertex $v$ also has a loop labelled by $d$ at every vertex. 
Moreover, by mapping to the maximal group image it is readily seen that 
these account for all the edges between vertices in this embedded copy of $\Gamma$, that is, the induced subgraph of $S\Gamma(1)$ on this set of vertices is isomorphic to $\Gamma$ with a loop labelled by $d$ added to every vertex.     

In particular, the end of the $z$-vertex starting at $1$ is a $p$-vertex, and therefore is the root of an embedded copy of $\Gamma$. Moreover, an easy inductive
argument using the penultimate relation in the presentation shows that every vertex in this copy of $\Gamma$ corresponding to an element of the subgroup $H$ is a $z$-vertex.
Since in the embedded copy of $\Gamma$ every edge labelled by $b \in B$ has a corresponding reverse edge labelled by $\overline{b} \in B$ it then follows that     
for every word $w$ over $B$ (including the empty word), $zwz'$ can be read both into and out of the start vertex of $\sgo$ which means that these
elements represent units of the inverse monoid $M$. 
Furthermore, it is easy to see that for any two such words $zuz'$ and $zvz'$ we have that $(zuz')(zvz') = z(uv)z'$ in $M$ (by Lemma~\ref{lem_reducerightinvert} since $zuz'$ and $zvz'$ are both right invertible in $M$), 
and $zuz' = zvz'$ in $M$ if and only
if $u = v$ in $H$.
For the non-trivial direction of the latter claim, if $zuz' = zvz'$ in $M$ then
$zuz' = zvz'$ in the maximal group image which implies $u = v$ in the maximal group image 
of $M$ 
and thus $u=v$ in $H$. 
It follows that 
the set of elements of $M$ represented by the set of words $\{ zwz' : w \in B^* \}$  
forms a subgroup of the group of units of $M$ that is isomorphic to the
group $H$.
Call this subgroup $C$.

To complete the proof that the group of units of $M$ is isomorphic to $H$ it    
suffices to check that $M$ has no units other than those in $C$. By 
\cite[Theorem 1.3]{GrayRuskucUnits} and \cite[Proposition 4.2]{Ivanov:2001kl} 
the relators in a special inverse monoid can be factorised into subwords representing units, in such a way that the factors generate the entire group of units. So it suffices to check that if a relator can be written $uvw$ where $u,v$ and $w$ represent units then $v$ is in $C$. First notice that if we add relations to the presentation identifying $d$ and the generators in $A$ with $1$, we obtain a natural
morphism from $M$ onto the monoid:
\begin{align*}
N = \mathrm{Inv}\langle \ p_0, \dots, p_k, \ z \ \mid \ &p_i p_i' = 1 &(i = 0, \dots, k) \\
&zz' = 1 \\
&z \left( \prod_{i=0}^k p_i' p_i \right) z' = 1 \ \rangle
\end{align*}
Notice that $N$ itself has trivial group of units: indeed, it has a homomorphism onto the bicyclic monoid $B =
\mathrm{Inv}\langle b \mid bb' = 1 \rangle$ taking $z$ to $b^2$ and all the $p_i$s
to $b$. Any unit in $N$ must map to $1$ under this map. If there are non-trivial units there must be a proper prefix of a relator which maps to $1$ under this
map, but it is easy to see that all prefixes of relators map to $b$ or $b^2$.
Note that this argument in particular shows that none of the generators $\{p_0, p_1, \ldots, p_k, z\}$ is invertible in $N$.    

Clearly units in $M$ must map to units in $N$; so units in $M$ must map to $1$ in $N$.
Thus, it suffices to consider factorisations of $uvw$ of relators (with $u$ or $w$ possibly empty) such that $u$, $v$ and $w$ map to $1$ in $N$, and show that in such cases $v$ either is not a unit or is necessarily in $C$. 
The only such factorisations of relators 
in the presentation of $M$ 
are (i) the factorisations $(p_i a p_i') (p_i a' p_i')$ of the first type of relation, and (ii) the factorisations $(zbz')(zb'z')$ of the penultimate type. 
That these are the only such factorisations can be proved using fact that none of the generators $\{p_0, p_1, \ldots, p_k, z\}$ is invertible in $N$, and hence certainly none of them or their inverses are equal to $1$ in $N$.    
In case (ii) we can see that the factors are already in $C$, so to complete the proof of the theorem it suffices to show that $p_i a p_i'$ is not a unit for
each $a \in A$ and $i \in I_0$.

To this end, 
we define an infinite graph $\Omega$ recursively as follows. The graph $\Omega$ has a root vertex at which for each $x \in \lbrace z, p_0, \dots, p_i \rbrace$
there is an $x$-edge going out of the vertex with a copy of $\Omega_x$ at the far end, where, $\Omega_x$ is defined as follows
For each $i \in \lbrace 0, 1, \dots, k \rbrace$, a \textit{$p_i$-zone} $\Omega_{p_i}$ consists of a copy of the free monoid Cayley graph on $A$ (rooted at $1$) with
\begin{itemize}
\item at each vertex and for each $x \in \lbrace z, p_0, \dots, p_k \rbrace$, an $x$-edge going out with a copy of $\Omega_x$ at the far end;
\item at each vertex (except the root) an edge coming in labelled $p_i$, with a copy of $\Omega_{p_i'}$ at the far end;
\item if $i\geq 1$, at each vertex $v$ and for each relator $r_i$ an edge labelled $d$ from $v$ to the end of the path starting at $v$ and labelled $r_i$; and
\item if $i=0$,  at each vertex $v$ a loop labeled $d$.
\end{itemize}
A \textit{$z$-zone} $\Omega_{z}$ consists of a copy of the Cayley graph $\Gamma$ of the group $G$ with respect to the generating set $A$ with      
\begin{itemize}
\item at each vertex and for each $x \in \lbrace z, p_0, \dots, p_k \rbrace$ an $x$-edge going out with a copy of $\Omega_x$ at the far end;
\item at each vertex and for each $x \in \lbrace p_0, \dots, p_k \rbrace$ an $x$-edge coming in with a copy of $\Omega_{x'}$ at the far end;
\item at each vertex 
of the copy of $\Gamma$ 
corresponding to an element of $H$, a $z$-edge coming in with a copy of $\Omega_{z'}$ at the far end; and
\item at each vertex a loop labelled $d$.
\end{itemize}

For $x \in \lbrace z, p_0, \dots, p_k \rbrace$, an \textit{$x'$-zone} $\Omega_{x'}$ consists of a root vertex with 
\begin{itemize}
\item for each $y \in \lbrace z, p_0, \dots, p_k \rbrace$ with $y \neq x$ an $y$-edge going out with a copy of $\Omega_y$ at the far end.
\end{itemize}
See Figure~\ref{fig_pictorial2} for an illustration of the graph $\Omega$.

%
% Picture of the graph \Omega constructed in the proof below 
%
\begin{figure}
\begin{center}
\begin{tikzpicture}[scale=.3,  
TRectangle/.style ={draw, rectangle, thick, minimum height=8em, minimum width=8em}]
\filldraw (-4,0) circle (7pt);
%% Central square  
\draw[thick, fill=gray!20] (0,0) -- (8,0) -- (0,8) -- (0,0);
\draw[thick] (8,0) -- (8,8) -- (0,8);
%% Top left square and then work around clockwise all the small squares  
%%
%% Triangle of A^* on the left   
%
\draw[thick] (-10,2) -- (-10-4, 2+8) -- (-10+4, 2+8) -- (-10,2);
%
% Labels for central square 
%
\node at (2,4) {\scriptsize $\Gamma_{H}$};
\node at (3.5,7) {\scriptsize $\Gamma \setminus \Gamma_{H}$};
%%
%
% Labels for triangle on the left  
%
\node at (-10,7) {\scriptsize $A^*$};
\node at (-10.3,13) {\scriptsize $z$};
\node at (-8.7,13) {\scriptsize $p_0$};
\node at (-5.4,13) {\scriptsize $p_k$};
\node at (-7.1,13) {\scriptsize $\ldots$};
\node at (-3.4,8.6) {\scriptsize $p_i$};
%%
%%
%% Labels below the triangle on the left 
%%
\node at (-8,-0.5) {\scriptsize $p_0$};
\node at (-9,1.2) {\scriptsize $\vdots$};
\node at (-8,2.3) {\scriptsize $p_i$};
\node at (-6.2,2.6) {\scriptsize $\iddots$};
\node at (-6.2,4.3) {\scriptsize $p_k$};
%%
%%
%% Main top right arrows 
\arrowdraw{0.6}{6,6}{4,12};
\arrowdraw{0.6}{6,6}{6,12};
\arrowdraw{0.6}{6,6}{8,12};
\arrowdraw{0.5}{12,8}{6,6};
\arrowdraw{0.5}{12,6}{6,6};
\arrowdraw{0.5}{12,4}{6,6};
%%
%% Main bottom left arrows  
\arrowdraw{0.6}{2,2}{4,-4};
\arrowdraw{0.6}{2,2}{2,-4};
\arrowdraw{0.6}{2,2}{0,-4};
\arrowdraw{0.5}{-4,4}{2,2};
\arrowdraw{0.5}{-4,2}{2,2};
\arrowdraw{0.5}{-4,0}{2,2};
\arrowdraw{0.6}{-4,0}{-10,0};
\arrowdraw{0.6}{-4,0}{-10,2};
\arrowdraw{0.6}{-4,0}{-6,6};
\arrowdraw{0.6}{-8,8}{-10,14};
\arrowdraw{0.6}{-8,8}{-8,14};
\arrowdraw{0.6}{-8,8}{-6,14};
\arrowdraw{0.5}{-2,8}{-8,8};
\arrowdraw{0.6}{12,6}{18,8};
\arrowdraw{0.6}{12,6}{18,6};
\arrowdraw{0.6}{12,6}{18,4};
\arrowdraw{0.6}{12,6}{14,0};
%
%%
%
% Some labels 
%
%
\node at (3.8,11) {\scriptsize $z$};
\node at (5.2,11) {\scriptsize $p_0$};
\node at (6.9,11) {\scriptsize $\ldots$};
\node at (4.2+4.3,11) {\scriptsize $p_k$};
%
% Labels at bottom of picture 
%
\node at (-0.2,-3) {\scriptsize $z$};
\node at (1.4,-3) {\scriptsize $p_0$};
\node at (2.8,-3) {\scriptsize $\ldots$};
\node at (0.3+4.2,-3) {\scriptsize $p_k$};
\node at (-2,4) {\scriptsize $p_k$};
\node at (-3,3.2) {\scriptsize $\vdots$};
\node at (-2,1.4) {\scriptsize $p_0$};
\node at (-2,3.8-3.6) {\scriptsize $z$};
\node at (10.5,8.2) {\scriptsize $p_0$};
\node at (10.5,7.1) {\scriptsize $\vdots$};
\node at (11.4,6.5) {\scriptsize $p_i$};
\node at (10.5,5.6) {\scriptsize $\vdots$};
\node at (10.5,3.9) {\scriptsize $p_k$};
\node at (16.5,8) {\scriptsize $z$};
\node at (16.5,6.5) {\scriptsize $p_0$};
\node at (16.5,5.5) {\scriptsize $\vdots$};
\node at (19.2,3.5) {\scriptsize $p_j \; \; (j \neq i)$};
\node at (15,3) {\scriptsize $\iddots$};
\node at (12.9,1) {\scriptsize $p_k$};
\end{tikzpicture}
\end{center}
\caption{
An illustration of part of the graph $\Omega$ constructed in the proof of Theorem~\ref{thm_possibleunits}.  
The black vertex is the root, 
the square in the centre represents a $z$-zone $\Omega_z$ with the copy of $\Gamma$ partitioned into vertices $\Gamma_H$ representing elements of $H$ and its complement, 
while the triangle on the left represents a $p_i$-zone $\Omega_{p_i}$ where the $d$-labelled edges in the $p_i$-zone have been omitted from the diagram.     
}\label{fig_pictorial2}
\end{figure}
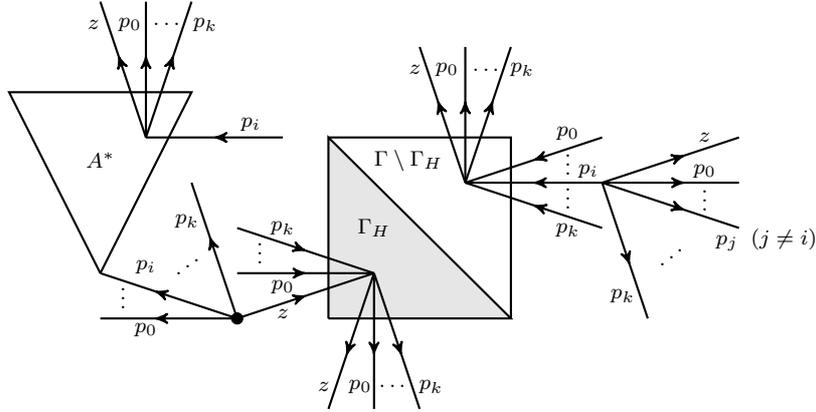

It is immediate from the definition of the zones that $\Omega$ is bi-deterministic, since we were careful to ensure that there are never two edges with the
same label leaving the same vertex, and that where we attach a $\Omega_x$ at the end of a new edge, it is of a type whose root has no existing $x$-edge
coming in.

Next we claim that every relator can be read around a closed path at every vertex. First notice that every vertex is the start of a $z$-edge leading into a $z$-zone, and a
$p_i$ edge leading into either a $p_i$-zone or a $z$-zone (not necessarily at the root). With this observation it is easy to check that the defining relators which are fundamental idempotents
can all be read (necessarily along a closed path, since the graph is bi-deterministic) at every vertex.

For the relators of the form $p_i r_i d' p_i'$, notice that reading the initial $p_i$ from any vertex will take us into either a $p_i$-zone or a $z$-zone. If it is a $p_i$-zone then $r_i$ can be read in the free monoid Cayley graph, and we have inserted a $d$-edge from there back to the point at which we came in. If it is a $z$-zone then since
$r_i$ represents $1$ in $G$, it can be read around a closed path, and $d$ can then be read around a loop. In both cases, the original $p_i'$ can then be read back
along the original edge to the start point. Similarly, for the relators of the form $p_0 d p_0'$, reading $p_0$ from any vertex takes us to a $p_0$-zone or a $z$-zone, whereupon
we can read $d$ around a loop and then $p_0'$ back to the start.

Since 
$\Omega$ is bi-deterministic and 
in $\Omega$ every relator can be read around a closed path at every vertex 
it follows from Lemma~\ref{lem_graphmorphisms} 
that there is a morphism from the Sch\"utzenberger graph $S \Gamma(1)$ to $\Omega$. Notice for each $i$, that the vertex which $p_i$ maps
to under this morphism is the root of a $p_i$-zone, and therefore for each $a \in A$ has no $a$-edge coming in. It follows that $a p_i'$ cannot be read into the root
of $\Omega$, and hence $a p_i'$ cannot be read into the root of $S\Gamma(1)$. Therefore, the word $a p_i'$ is not left invertible, which means that  
$p_i a p_i'$ 
cannot
be a unit. This completes the proof.
\end{proof}

We remark that in fact the graph $\Omega$ constructed in the proof of Theorem~\ref{thm_possibleunits} is isomorphic (via the given morphism) to
$S\Gamma(1)$. To show this requires more work and is not needed for the above proof of the theorem, but an alternative proof of the
theorem could be obtained by establishing this fact, observing that the automorphism group of $\Omega$ is
easily seen to be $H$, and using the fact that the group of units in an inverse monoid is always the automorphism
group of $S \Gamma(1)$. In fact this viewpoint provided the intuition from which we arrived at the presentation used in the proof.

We also remark on the high proportion of the proof which is devoted to showing that we do not have extra ``undesired'' group elements. Proving that
unwanted things do \textit{not} live in particular subgroups seems to be by some distance the hardest problem when working in this area, in the same way that
proving certain words are \textit{not} equal is often the hardest part of working with a group or monoid presentation.

It remains an open question which groups arise as groups of units of \textit{one-relator} special inverse monoids. Theorem~\ref{thm_possibleunits} does not really
provide any help with this particular question, since even after eliminating fundamental idempotent relators using Lemma~\ref{lemma_removeidempotents}, the number of relations in the
presentation constructed will be one more than the number of relations in a \textit{special monoid} presentation for the underlying group, which may already be more than the number of relations in a group presentation. We shall see below (Corollary~\ref{cor_onerelator}) that every finitely generated subgroup of a one-relator group arises as a group $\GreenH$-class of a one-relator special inverse monoid, but we do not know if all such groups arise as groups of units of one-relator special inverse monoids. We also do not know whether
groups of units in one-relator special inverse monoids are necessarily finitely presented. It was shown in \cite{GrayRuskucUnits} that a positive answer to the latter question would imply that all one-relator groups are coherent, resolving a long-standing open problem first posed by G.~Baumslag \cite[page 76]{Baumslag1974}. We refer to the recent survey article of Wise \cite{WiseCoherence2020} for more background on this and other aspects of the theory of coherent groups.

\section{Group $\GreenH$-classes}\label{sec_GroupHClasses}

In the previous section we exactly characterised possible groups of units of finitely presented special inverse monoids; these are (as was already known) all
finitely generated, and it is natural to ask whether the same is true for the other group $\GreenH$-classes in such a monoid. Indeed, in the case of special (non-inverse) monoids it is even known that the other group $\GreenH$-classes are all isomorphic to the group of units, so one might even ask if this is the case for special
inverse monoids. In section we shall show that the answer to both questions is negative: it transpires that every recursively
presented group (not just the finitely generated ones) can arise as a group $\GreenH$-class in a finitely presented special inverse monoid! The same construction will also allow us to show that a finitely
presented special inverse monoid can have infinitely many pairwise non-isomorphic group $\GreenH$-classes.

Our construction will need some preliminary results, starting with the following straightforward group-theoretic fact:
\begin{lemma}\label{lemma_stabiliser}
Let $K$ be a subgroup of a group $G$, and $t \in G$. Then the
setwise stabiliser of $K \cup tK$ under the left translation action of $G$,
$$\lbrace g \in G \ \mid \ g(K \cup tK) = K \cup tK \rbrace,$$
is either equal to, or an index $2$ overgroup of, $K \cap tKt^{-1}$, with equality provided $K \cap tKt = \emptyset$.
\end{lemma}
\begin{proof}
Let
$$H = \lbrace g \in G \mid gK = K \textrm{ and } gtK = tK \rbrace$$
be the intersection of the setwise stabilisers of $K$ and $tK$. We claim that $H = K \cap tKt^{-1}$.
Indeed, if $g \in K \cap tKt^{-1}$ then for every $h \in K$ we have $gh \in K$, while writing $g = tkt^{-1}$ with $k \in K$ we have
$gth = tkt^{-1}th = t(kh) \in tK$, so that $g \in H$. Conversely, if $g \in H$ then $gK = K$ and $gtK = tK$; the first equation gives $g \in K$ and the second gives $t^{-1}gtK = K$ so that $t^{-1}gt \in K$ and $g \in K \cap tKt^{-1}$.

Now if $g(K \cup tK) = K \cup tK$ then because left cosets form a partition either (i) $gK = K$ and $gtK = tK$
or (ii) $gK = tK$ and $gtK = K$. It follows that there is a morphism from the stabiliser of $K \cup tK$ to the cyclic group $\mathbb{Z}_2$, taking $g$ to $0$ in case (i) and
$1$ in case (ii). By definition the kernel of this morphism is $H = K \cap tKt^{-1}$, so the latter has index $1$ or $2$ in the stabiliser, depending on whether
the morphism is trivial or surjective.

Moreover, if $K \cap tKt = \emptyset$ then case (ii) cannot arise. Indeed, if $gK = tK$ and $gtK = K$ then the first equation gives $g = th$ for some $h \in K$, whereupon the second gives $tht \in K$; but now $tht$ lies in $K \cap tKt$, contradicting our hypothesis that this intersection is empty. Thus, in this case the morphism is trivial and
we have equality.
\end{proof}

\begin{lemma}\label{lemma_usefulgroup}
Let $K$ be a finitely presented group, and $H$ a recursively enumerable subgroup of $K$. Then there exists a finitely
presented group containing two conjugate subgroups $K_1$ and $K_2$ such that $K_1 \cong K_2 \cong K$ and $K_1 \cap K_2 \cong H$. Moreover, this
group can be chosen in such a way that $t K_1 t^{-1} = K_2$ for some $t$ such that
$K_1 \cap t K_1 t = \emptyset$.
\end{lemma}
We note that a subgroup of a finitely presented group is recursively enumerable if and only if it is generated by a recursively enumerable set.
\begin{proof}
Let $G_1$ be an HNN extension of $K$, with stable letter $t$ conjugating each element of $H$ to itself. Since $K$ is finitely presented and $H$ is recursively enumerable, $G_1$ is finitely generated and recursively presented by a presentation consisting of the presentation for $K$ with an extra generator $t$ and extra relations $twt^{-1} = w$ for each word $w$ representing an element of $H$. So by the Higman embedding theorem, $G_1$ can be embedded in a finitely presented group $G_2$. But now $G_2$ contains two conjugate subgroups $K_1 = K$ and $K_2 = tKt^{-1}$, which intersect exactly in a copy of $H$, as required. Finally, it follows easily from Britton's Lemma \cite[page 181]{Lyndon:2001lh} that $K_1 \cap t K_1 t$ is empty in $G_1$ and hence also in $G_2$.
\end{proof}

\begin{theorem}\label{thm_stabiliser}
Let $G$ be a group, and $H$ a subgroup of $G$. Then there exists an $E$-unitary special inverse monoid $M$ such 
that for every \textbf{finite} subset $X \subseteq G$, the setwise stabiliser in $G$ of the union of the left $H$-cosets determined by $X$,
$$K = \lbrace g \in G \mid gXH= XH \rbrace,$$
is isomorphic to a group $\GreenH$-class in $M$. If $G$ is finitely presented and $H$ is finitely generated then $M$ can be chosen to be finitely
presented with the same number of defining relations as $G$ (unless $G$ is free, in which case $M$ will have $1$ defining relation). If $H$ is trivial then
$M$ can be chosen to be the free product of $G$ with a free inverse monoid.
\end{theorem}
\begin{proof}

Let $\mathrm{Gp}\langle A \mid R \rangle$ be a group
 presentation for $G$, choosing $A$ to be finite if $G$ is finitely generated and $R$ to be finite if $G$ is finitely presented. If $G$ is free choose $A$ to be a free
 generating set and $R$ to be empty. Let $B$ be a set of words over $A^{\pm 1}$ representing a monoid generating set for $H$, choosing $B$ to be empty if $H$ is trivial and finite if $H$ is finitely generated. Let $\Gamma$ be the Cayley graph of $G$ with respect to $A$.

Let $y$ and $z$ be new symbols not in $A$, and define
\begin{align*}
M \ = \ \mathrm{Inv}\langle A, y, z \ \mid \ & R,\  aa' = a'a=1 \ (a \in A), \\
&zbz'zb'z'=yby'yb'y'=1 \ (b \in B) \rangle.
\end{align*}
Notice that the presentation is finite provided $G$ is finitely presented 
and $H$ is finitely generated, 
and clearly presents the free product $G * \mathrm{Inv} \langle  y, z \rangle$ if $H$ is trivial (since we chose $B$ to be empty in that case). In the case that the presentation is finite, since all the relators except those in $R$ are fundamental idempotents, repeated application of Lemma~\ref{lemma_removeidempotents} will give a special inverse monoid presentation for $M$ with $|A|+2$ generators and $|R|$ relations (or $1$ relation if $R$ is empty, so that $G$ is free).

Because the relators in $R$ do not feature the letters $y$ and $z$ and the relators not in $R$ are all fundamental idempotents. It follows easily that this presentation when interpreted as a group presentation yields
a free product $G * \mathrm{Gp}\langle y, z \rangle$ of $G$ with a free group of rank $2$, so this free product is the maximal group image of $M$. In particular, it
follows that $A$ generates a copy of $G$ inside the group of units of $M$. 

To see that $M$ is $E$-unitary first note that
$$N \ = \ \mathrm{Inv}\langle A, y, z \ \mid R, \ aa'=a'a=1 \ (a \in A) \rangle \ = \ G * \mathrm{Inv} \langle  y,z \rangle$$
is the inverse monoid free product of the group $G$ and a free inverse monoid, both of which are $E$-unitary, and hence $G *  \mathrm{Inv} \langle  y,z \rangle$ is $E$-unitary. The fact that the free product of two $E$-unitary inverse monoids is again $E$-unitary can be proved, for example, by applying a result of Stephen \cite[Theorem 6.5]{Stephen98} that gives sufficient conditions for the amalgamated free product of two $E$-unitary inverse semigroups to be $E$-unitary; see \cite[Proof of Theorem 3.8]{Gray2020} for details. Hence, $N$ is $E$-unitary, and since the relators of the form $zbz'zb'z'=yby'yb'y'=1$ hold in every group, it follows from Lemma~\ref{lemma_eunitarytransfer} that $M$ is $E$-unitary.

Now let $X$ be a finite set of words over the generators, representing a finite subset of $G$, and consider the element
$$e = \prod_{x \in X} xz'zy'yx',$$
noting that the order of the product is unimportant, and the element $e$ is idempotent, because the factors are fundamental idempotents and idempotents commute in inverse semigroups. We claim that the $\GreenH$-class of $e$ is isomorphic to $K$.

To this end, consider the Schutzenberger graph $S \Gamma(e)$. First, we note that it contains a copy of the Cayley graph $\Gamma$, rooted at $e$. Indeed, any
word over $A^{\pm 1}$ is right invertible in $M$ (because of the relations of the form $aa'=a'a=1$) and hence readable from $e$, and two words representing the same element in $G$ are equal in $M$ (because of the relators in $R$ and those of the form $aa'=a'a=1$), so that the corresponding paths from $e$ must end in the same place. Conversely, if two words $u$ and $v$ over $A^{\pm 1}$ label paths from
$e$ ending in the same place then $eu = ev$ in $M$, and hence also in the maximal group image, which means that $u=v$ in the maximal group image $G \ast \mathrm{Gp}\langle y,z \rangle$ and hence $u=v$ in $G$.

We claim next that the vertices of $S \Gamma(e)$ corresponding to elements of $XH \subseteq G$ are exactly those which have both a $y$-edge and a $z$-edge coming in. Indeed, suppose $x \in X$
and $h \in H$ so that $xh \in XH$. 
Since the word defining $e$ can be read from the root of $S\Gamma(e)$, the word $xz'zy'yx'$ (which occurs as a factor of $e$ after a fundamental idempotent prefix) can be read from $e$, so the vertex corresponding to 
$x$ has both a $y$-edge (call it $e$) and a $z$-edge coming in. 
In particular in the case $h=1$ this shows that $xh=x$ does have both a $y$-edge and a $z$-edge coming in. If $h \neq 1$ then $h$ is a product
of elements of our chosen generating set for $H$, so is represented by a word of the form $b_1 \dots b_k$ where each $b_1, \dots, b_k \in B$.  
Now from the start of the $y$-edge $e$ we can read the (right invertible) word $(y b_1 y') (y b_2 y') \dots (y b_k y')$,
and it is easy to see that the final letter of this word must be read along a $y$-edge coming into the vertex corresponding to $xh$. An identical argument shows
that there is a $z$-edge coming into this vertex.

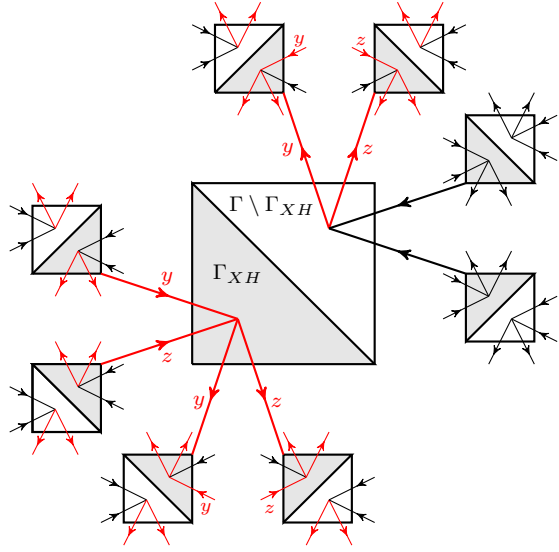
\begin{figure}
\begin{center}
\begin{tikzpicture}[scale=.3, 
TRectangle/.style ={draw, rectangle, thick, minimum height=8em, minimum width=8em}],
\draw[thick, fill=gray!20] (0,0) -- (8,0) -- (0,8) -- (0,0);
\draw[thick] (8,0) -- (8,8) -- (0,8);
\draw[thick, fill=gray!20] (1,12) -- (4,12) -- (4,15) -- (1,12);
\draw[thick] (1,12) -- (1,15) -- (4,15);
\draw[thick, fill=gray!20] (8,12) -- (11,12) -- (8,15) -- (8,12);
\draw[thick] (11,12) -- (11,15) -- (8,15);
\draw[thick, fill=gray!20] (12,8) -- (15,8) -- (12,11) -- (12,8);
\draw[thick] (15,8) -- (15,11) -- (12,11);
\draw[thick, fill=gray!20] (12,4) -- (15,4) -- (12,1) -- (12,4);
\draw[thick] (12,1) -- (15,1) -- (15,4);
\draw[thick, fill=gray!20] (4,-4) -- (7,-4) -- (4,-7) -- (4,-4);
\draw[thick] (7,-4) -- (7,-7) -- (4,-7);
\draw[thick, fill=gray!20] (0,-4) -- (0,-7) -- (-3,-4) -- (0,-4);
\draw[thick] (0,-7) -- (-3,-7) -- (-3,-4);
\draw[thick, fill=gray!20] (-4,0) -- (-4,-3) -- (-7,0) -- (-4,0);
\draw[thick] (-7,0) -- (-7,-3) -- (-4,-3);
\draw[thick, fill=gray!20] (-4,0) -- (-4,-3) -- (-7,0) -- (-4,0);
\draw[thick] (-7,0) -- (-7,-3) -- (-4,-3);
\draw[thick, fill=gray!20] (-4,4) -- (-4,7) -- (-7,4) -- (-4,4);
\draw[thick] (-7,4) -- (-7,7) -- (-4,7);
%%
%% 
%% Arrows 
%%
%% Main top right arrows 
\redarrowdraw{0.6}{6,6}{4,12};
\redarrowdraw{0.6}{6,6}{8,12};
\arrowdraw{0.5}{12,8}{6,6};
\arrowdraw{0.5}{12,4}{6,6};
%%
%% Main bottom left arrows  
\redarrowdraw{0.6}{2,2}{4,-4};
\redarrowdraw{0.6}{2,2}{0,-4};
\redarrowdraw{0.5}{-4,4}{2,2};
\redarrowdraw{0.5}{-4,0}{2,2};
%%
%% Small squares starting with the let most of the op set and 
%% then working around all the small squares clockwise
%%
\smallarrowdraw{0.4}{0,13}{2,14};
\smallarrowdraw{0.4}{0,15}{2,14};
\smallredarrowdraw{0.8}{2,14}{1,16};
\smallredarrowdraw{0.8}{2,14}{3,16};
\smallredarrowdraw{0.8}{3,13}{2,11};
\smallredarrowdraw{0.8}{3,13}{4,11};
\smallarrowdraw{0.4}{5,12}{3,13};
\smallredarrowdraw{0.4}{5,14}{3,13};
\smallarrowdraw{0.4}{7,12}{9,13};
\smallredarrowdraw{0.4}{7,14}{9,13};
\smallredarrowdraw{0.8}{9,13}{8,11};
\smallredarrowdraw{0.8}{9,13}{10,11};
\smallredarrowdraw{0.8}{10,14}{9,16};
\smallredarrowdraw{0.8}{10,14}{11,16};
\smallarrowdraw{0.4}{12,13}{10,14};
\smallarrowdraw{0.4}{12,15}{10,14};
\smallarrowdraw{0.4}{11,10}{13,9};
\smallarrowdraw{0.4}{11,8}{13,9};
\smallarrowdraw{0.8}{13,9}{12,7};
\smallarrowdraw{0.8}{13,9}{14,7};
\smallarrowdraw{0.8}{14,10}{13,12};
\smallarrowdraw{0.8}{14,10}{15,12};
\smallarrowdraw{0.4}{16,9}{14,10};
\smallarrowdraw{0.4}{16,11}{14,10};
\smallarrowdraw{0.4}{11,2}{13,3};
\smallarrowdraw{0.4}{11,4}{13,3};
\smallarrowdraw{0.8}{13,3}{12,5};
\smallarrowdraw{0.8}{13,3}{12+2,5};
\smallarrowdraw{0.4}{14+2,2+1}{14,2};
\smallarrowdraw{0.4}{14+2,2-1}{14,2};
\smallarrowdraw{0.8}{14,2}{14+1,2-2};
\smallarrowdraw{0.8}{14,2}{14-1,2-2};
\smallredarrowdraw{0.4}{5-2,-5-1}{5,-5};
\smallarrowdraw{0.4}{5-2,-5+1}{5,-5};
\smallredarrowdraw{0.8}{5,-5}{5-1,-5+2};
\smallredarrowdraw{0.8}{5,-5}{5+1,-5+2};
\smallredarrowdraw{0.8}{6,-6}{6-1,-6-2};
\smallredarrowdraw{0.8}{6,-6}{6+1,-6-2};
\smallarrowdraw{0.4}{6+2,-6-1}{6,-6};
\smallarrowdraw{0.4}{6+2,-6+1}{6,-6};
\smallredarrowdraw{0.8}{-1,-5}{-1-1,-5+2};
\smallredarrowdraw{0.8}{-1,-5}{-1+1,-5+2};
\smallarrowdraw{0.4}{-1+2,-5+1}{-1,-5};
\smallredarrowdraw{0.4}{-1+2,-5-1}{-1,-5};
\smallarrowdraw{0.4}{-2-2,-6+1}{-2,-6};
\smallarrowdraw{0.4}{-2-2,-6-1}{-2,-6};
\smallredarrowdraw{0.8}{-2,-6}{-2-1,-6-2};
\smallredarrowdraw{0.8}{-2,-6}{-2+1,-6-2};
\smallarrowdraw{0.4}{-2-2-4,-6+1+4}{-2-4,-6+4};
\smallarrowdraw{0.4}{-2-2-4,-6-1+4}{-2-4,-6+4};
\smallredarrowdraw{0.8}{-2-4,-6+4}{-2-1-4,-6-2+4};
\smallredarrowdraw{0.8}{-2-4,-6+4}{-2+1-4,-6-2+4};
\smallredarrowdraw{0.8}{-1-4,-5+4}{-1-1-4,-5+2+4};
\smallredarrowdraw{0.8}{-1-4,-5+4}{-1+1-4,-5+2+4};
\smallarrowdraw{0.4}{-1+2-4,-5+1+4}{-1-4,-5+4};
\smallarrowdraw{0.4}{-1+2-4,-5-1+4}{-1-4,-5+4};
\smallarrowdraw{0.4}{-6-2,6-1}{-6,6};
\smallarrowdraw{0.4}{-6-2,6+1}{-6,6};
\smallredarrowdraw{0.8}{-6,6}{-6-1,6+2};
\smallredarrowdraw{0.8}{-6,6}{-6+1,6+2};
\smallarrowdraw{0.4}{-5+2,5+1}{-5,5};
\smallarrowdraw{0.4}{-5+2,5-1}{-5,5};
\smallredarrowdraw{0.8}{-5,5}{-5-1,5-2};
\smallredarrowdraw{0.8}{-5,5}{-5+1,5-2};
%
%
% Some  labels 
%
\node at (2,4) {\scriptsize $\Gamma_{XH}$};
\node at (3.5,7) {\scriptsize $\Gamma \setminus \Gamma_{XH}$};
\node at (4.2,9.5) {\color{red} \scriptsize $y$};
\node at (4.2+3.5,9.5) {\color{red} \scriptsize $z$};
\node at (4.7,14.3) {\color{red} \tiny $y$};
\node at (4.7+2.7,14.3) {\color{red} \tiny $z$};
\node at (0.3,-1.7) {\color{red} \scriptsize $y$};
\node at (0.3+3.4,-1.7) {\color{red} \scriptsize $z$};
\node at (0.6,-6.4) {\color{red} \tiny $y$};
\node at (0.6+2.8,-6.4) {\color{red} \tiny $z$};
\node at (-1.1,3.8) {\color{red} \scriptsize $y$};
\node at (-1.1,3.8-3.5) {\color{red} \scriptsize $z$};
\end{tikzpicture}
\end{center}
\caption{
An illustration of part of the graph $S\Gamma(e)$ within the Cayley graph $\Delta$ of the maximal group image $G \ast \mathrm{Gp}\langle z, y \rangle$. 
Each square in the figure is a copy of the Cayley graph $\Gamma$ of $G$, each partitioned into the vertices $\Gamma_{XH}$ corresponding to the elements $XH$, depicted as shaded regions in the figure, and the complement of this set. 
The graph $S\Gamma(e)$ contains all the red edges and 
has vertex set equal to the union of all the copies $\Gamma$ to which these red edges are incident.    
The graph $\Delta'$ is the non-bi-deterministic quotient of $\Delta$ with vertices the copies of $\Gamma$ (i.e. the squares) and all $y$- and $z$-edges between them.
The graph $S\Gamma(e)'$ is the corresponding quotient of $S\Gamma(e)$.   
}\label{fig_pictorial}
\end{figure}

For the converse, since $M$ is $E$-unitary 
by Lemma~\ref{lem_fullSubgraph}
we may consider $S \Gamma(e)$ as a subgraph of the Cayley graph $\Delta$ of the maximal group image, which we have 
already observed is a free product of $G$ (generated by $A$) with a free group on $y$ and $z$.  This Cayley graph $\Delta$  partitions into a ``tree'' of copies of the Cayley
graph of $G$, in other words, a collection of copies of $\Gamma$ (which we shall call \textit{components}), joined by cut edges labelled $y$ and $z$. 
See Figure~\ref{fig_pictorial} for an illustration of the graph $\Delta$.  
Let $\Delta'$
be the quotient graph whose vertices are the components, and with an edge between two components exactly if $\Delta$ has an edge between two vertices in the
respective components with the same label. (Note that $\Delta'$ is not bi-deterministic; indeed it will typically have many edges leaving with the same label.)
Let $S\Gamma(e)'$ be the corresponding quotient of $S\Gamma(e)$.\footnote{
In general the precise structure of $S\Gamma(e)'$ will depend on the cardinalities of $G$ and $XH$ (which if infinite has the same cardinality as $H$ since $X$ is finite). If $\kappa$ is the cardinality of $G$ and $\mu$ is that of $XH$ then 
$\Delta'$ will be a tree where every vertex has 
$\kappa$ many in- and out-edges labelled by $y$ and also 
$\kappa$ many in- and out-edges labelled by $z$.  
The subgraph $S\Gamma(e)'$ will have a root vertex with $\kappa$ many out-edges labelled by each of $y$ and $z$, and $\mu$ many in-edges labelled by $y$ and $z$. 
Every vertex in $S\Gamma(e)'$ will have $\kappa$ many many out-edges labelled by each of $y$ and $z$. 
Any vertex with an in-edge labelled by $y$ will have $\mu$ many in-edges labelled by $y$, and   
any vertex with an in-edge labelled by $z$ will have $\mu$ many in-edges labelled by $z$. 
Furthermore $S\Gamma(e)'$ is a tree since it is a subgraph of the tree $\Delta'$.   
In particular the only vertex in $S\Gamma(e)'$ with both a $y$-edge and a $z$-edge coming in is the root vertex. 
It follows that all vertices of $S\Gamma(e)$ with both an $y$-edge and a $z$-edge coming in are within the component of $e$.
}
 We claim that the only vertex of $S\Gamma(e)'$ with both a $y$-edge
and a $z$-edge coming in is that corresponding to the component of $e$.

By Lemma~\ref{lem_fullSubgraph}, $S\Gamma(e)$ is equal to the smallest subgraph of the Cayley graph $\Delta$ such that $e$ can be read from $1$ and every relator  from the presentation for $M$ can be read at every vertex of $S\Gamma(e)$. 
It follows from this that $S \Gamma(e)'$ can be formed from the image in $\Delta'$ by iteratively tracing paths labelled by $yy'$ and $zz'$ in $\Delta'$, in other words, either adding a $y$-edge out of an existing vertex to a new vertex, or or adding a $y$-edge out of an existing vertex to a new vertex and another $y$-edge coming
into the new vertex, or adding a $y$-edge into an existing vertex which already has a $y$-edge into it, or corresponding operations with $z$-edges. Therefore, the only
way a vertex can acquire a $y$-edge coming in is if either (i) it is created as a new vertex at the end of a $y$-edge or (ii) it already had a $y$-edge coming in. This, any vertex
with a $y$-edge coming in (except for the component of $e$) must have been created with a $y$-edge coming in. The dual statement applies to $z$-edges. But
a vertex clearly cannot be created with both a $y$-edge and a $z$-edge coming in. 

Finally, suppose a vertex $v$ of $S \Gamma(e)$ within the component of $e$ has a $y$-edge coming in. Then by Lemma~\ref{lem_fullSubgraph} (or \cite[Lemma~3.1]{Stephen93}) the start $vy'$ of that edge must reachable from the vertices in the path traced out by $e$ in $\Delta$ (which means either a vertex in the component of $e$, or a vertex of the form $exy'$ or $exz'$ for some $x \in X$) by following a path labelled by a product of proper prefixes of the defining relators. Consider such a path where the number of
relator-prefixes is minimal. Consider the projection of $M$ onto $\mathbb{Z}$ taking all generators in $A$ to $0$ and the generators $y$ and $z$ to $1$. The entire component of $e$ clearly maps to $0$, and $exy'$, $exz'$ (for all $x \in X$) and $vy'$ map to $-1$. All relator-prefixes map to non-negative values, so the only way this can happen is if the path starts at $exy'$ or $exz'$ for some $x \in X$, and the relator-prefixes are of the form $yby'$ or $zbz'$ for $b \in B$. If we try to mix relator-prefixes of the form $yby'$ with those of the form $zbz'$ then the path clearly leaves the component of $e$, 
and can only return by eventually following $zb'z$ back to the same point, which contradicts the minimality of the number of relator-prefixes in the product. 
This can be seen e.g. by considering the structure of $S\Gamma(e)$ as a subgraph of the Cayley graph $\Delta$ of the maximal group image $G \ast \mathrm{Gp}\langle y, z \rangle$, as illustrated in Figure~\ref{fig_pictorial}, and observing that this fact is clearly true within the Cayley graph $\Delta$ of this free product and hence must also be true within the subgraph $S\Gamma(e)$.
Moreover, if we use only relator-prefixes of the form $zbz'$ we will clearly not end at $vy'$; so we must use relator-prefixes only of
the form $yby'$. By a similar argument, this means our path must start at $exy'$. So in summary, we have a path from $exy'$ to $vy'$ labelled by a product of words from $yBy'$. But this means there is a path from $ex$ to $v$ labelled by a product of words from $B$, which means that $v$ corresponds to an element of $XH$ as required.

Now let $\Omega$ be the graph which is the Cayley graph $\Gamma$ with additional $y$- and $z$-edges coming into the vertices corresponding to elements
of $XH$. Clearly, the automorphisms of $\Omega$ are exactly the automorphisms of $\Gamma$ which map elements of $XH$ to elements of $XH$, in other
words. The automorphisms all come from elements of $G$ acting by left translation, so this means the automorphism group of $\Omega$ is isomorphic to the
setwise stabiliser of $XH$ under the left translation action of $G$, in other words, to $K$.

Since $\Omega$ is a connected subgraph of
$S\Gamma(e)$ rooted at $e$ (in fact $\Omega$ is an induced subgraph of $S\Gamma(e)$), and since $e$ itself can be read in $\Omega$, it follows from Lemma~\ref{lem_StephenAutInvariant} that every automorphism of $\Omega$ will extend uniquely  
to an automorphism of $S\Gamma(e)$. On the other hand, any automorphism of $S\Gamma(e)$ must clearly preserve the set of vertices which have a both a $y$-edge
and a $z$-edge coming in, that is, the set of vertices corresponding to $XH$. It follows easily that it must preserve the embedded copy of $\Omega$, and hence
is an extension of an automorphism of $\Omega$. Thus,
$K$ is exactly the automorphism group of $S \Gamma(e)$, which by \cite[Theorem 3.5]{Stephen:1990ss} is isomorphic to the group $\GreenH$-class $H_e$ of $e$.\end{proof}

\begin{theorem}\label{thm_maxsubgroups}
Let $H$ be a recursively enumerable subgroup of a finitely presented group. Then there is an $E$-unitary finitely presented special inverse monoid with a
group $\GreenH$-class isomorphic to $H$.
\end{theorem}
\begin{proof}
By Lemma~\ref{lemma_usefulgroup} there is a finitely presented group $G$ with a finitely generated subgroup $K$ and an element $t$ such that
$K \cap tKt^{-1} \cong H$ and $K \cap tKt = \emptyset$. By Lemma~\ref{lemma_stabiliser} this means $H$ is isomorphic
to the stabiliser of $K \cup tK$. Now by Theorem~\ref{thm_stabiliser} there is a finitely presented special
inverse monoid with the latter as a group $\GreenH$-class.
\end{proof}
\begin{corollary}\label{cor:NonFGExample}
There exist $E$-unitary finitely presented special inverse monoids with group $\GreenH$-classes which are not finitely generated.
\end{corollary}

To give a concrete illustration of Corollary~\ref{cor:NonFGExample},
in the next result we shall show how to construct a specific example of a finitely presented special inverse monoid in which there is a maximal subgroup that is not finitely generated.

\begin{proposition} 
Let $M$ be the inverse monoid defined by the presentation 
\begin{align*}
\mathrm{Inv}\langle a, b, q, t, y, z \ \mid \ & 
aqa'q'=1, \; bqb'q'=1, \; tat'a'=1, \; tbt'b'q'=1,  \\
& cc' = c'c=1 \ (c \in \{a,b,q,t\}), \\
& zdz'zd'z'=ydy'yd'y'=1 \ (d \in \{a, a', b, b' \}) \rangle.
\end{align*}
Then $M$ is a finitely presented 
$E$-unitary 
special inverse monoid that contains a group $\gh$-class that is not finitely generated. Specifically, the $\gh$-class $H_e $ in $M$ of the idempotent   
$e = (z' z y' y)(t z' z y' y t')$ is not 
finitely generated. 
  \end{proposition}
\begin{proof}

We begin with the direct product $L$ of a free group of rank $2$ and the infinite cyclic group   
\[
L = \mathrm{Gp}\langle a,b \rangle \times \mathbb{Z} 
=
\mathrm{Gp}\langle a, b, q \mid
aq=qa, \;
bq=qb 
\rangle.  
\]
Let $H = \langle a,b \rangle \leq L$ be the free subgroup generated by $\{a,b\}$, 
and $U = \langle a, qb \rangle \leq L$
the subgroup generated by $\{a,qb\}$.    
It is immediate from the definitions that $H$ is isomorphic to the free group of rank $2$ with respect to the generating set $\{a,b\}$.     
It is straightforward to show that $U$ is also isomorphic to a free group of rank $2$ with respect to the generating set $\{a, qb\}$; indeed, observe that
\begin{align*}
& \; \; \; \; \, \mathrm{Gp}\langle a, b, q \mid
aq=qa, \;
bq=qb 
\rangle \\
&= \mathrm{Gp}\langle a, b, q, x \mid
x=qb, \;
aq=qa, \;
bq=qb 
\rangle  \\  
&= \mathrm{Gp}\langle a, x, q \mid
%x=qb, \;
aq=qa, \;
%%q^{-1}
xq=qx 
\rangle \cong   
\mathrm{Gp}\langle a,x \rangle \times \mathbb{Z}, 
\end{align*}
from which it is clear the subgroup generated by $\{x=qb, a \}$ is free of rank $2$ with respect to this generating set.   

Furthermore, it was proved by Moldavanski\u{\i} in \cite[Lemma~1]{Moldavanski68} that $H \cap U \leq L$ is not finitely generated. Indeed, the proof of \cite[Lemma~1]{Moldavanski68} shows that $H \cap U$ is equal to the normal closure of $a$ in the free group $\mathrm{Gp}\langle a,b \rangle$ which is shown to be a free group of infinite rank and hence not finitely generated.   

Since $H$ and $U$ 
are both isomorphic to the free group of rank $2$ 
we can form the HNN extension 
\[
G = 
\mathrm{Gp}\langle a, b, q, t \mid  
aq=qa, \;
bq=qb, \;
tat^{-1}=a, \;
tbt^{-1}=qb \rangle
\]
of $L$ with respect to the isomorphism $\phi:H \rightarrow U$ that maps $a \mapsto a$ and $b \mapsto qb$. 
Since $G$ is an HNN extension of $L$ it follows that $L$ naturally embeds in $G$ and hence so do $H$ and $U$. 
Hence $H = \langle a, b \rangle \leq G$ is a finitely generated subgroup $G$ that is isomorphic to the free group of rank $2$, and in $G$ we have that    
$
H \cap tHt^{-1} = H \cap \phi(H) = H \cap U 
$          
which is not finitely generated. 
Furthermore, it follows from Britton's Lemma applied to the HNN extension $L$ that   
$H \cap tHt = \varnothing$.  
Hence by Lemma~\ref{lemma_stabiliser} 
\[
K = \{ g \in G : g(H \cup tH) = H \cup tH \} = H \cap t H t^{-1}
\]
and hence in $G$ the stabiliser $K$ 
of the union of cosets 
$H \cup tH$ 
is not finitely generated.   

Now using this group $G$, the subgroup $H$, and the coset representatives $X = \{1_G, t \}$, it then follows from      
Theorem~\ref{thm_stabiliser} 
and its proof that 
that if 
$M$ is the 
inverse monoid defined by the presentation 
\begin{align*}
\mathrm{Inv}\langle \underbrace{a, b, q, t,}_{A} y, z \ \mid \ & 
\underbrace{aqa'q'=1, \; bqb'q'=1, \; tat'a'=1, \; tbt'b'q'=1}_{R},  \\
& cc' = c'c=1 \ (c \in \{a,b,q,t\}), \\
& zdz'zd'z'=ydy'yd'y'=1 \ (d \in \underbrace{\{a, a', b, b' \}}_{B} ) \rangle.
\end{align*}
then $M$ is a finitely presented 
$E$-unitary 
special inverse monoid that contains a group $\gh$-class that is not finitely generated. 
In particular, 
from $X = \{1_G, t \}$
we obtain the idempotent  
\[
e = (z' z y' y)(t z' z y' y t') 
\]
in the proof of 
Theorem~\ref{thm_stabiliser}
and then the maximal subgroup $H_e$ of $M$ is isomorphic to $K$ and hence is not finitely generated.  
\end{proof}

Another corollary of Theorem~\ref{thm_maxsubgroups} is the following.

\begin{corollary}\label{cor_possiblemaxsubgroups}
The possible group $\GreenH$-classes of finitely presented special inverse monoids (and of $E$-unitary finitely presented, recursively presented, and
$E$-unitary recursively presented special inverse monoids) are exactly the (not necessarily finitely generated) recursively presented groups.
\end{corollary}
\begin{proof}
Higman \cite[Corollary to Theorem 1]{Higman61} showed that every recursively presented group embeds ``effectively'' in a finitely presented group; although an
exact definition of ``effectively'' is not given, it is clear from the argument that this implies that the image will be a recursively enumerable subgroup of the finitely
presented group. Hence, by Theorem~\ref{thm_maxsubgroups}, every such group is a group $\GreenH$-class of some $E$-unitary finitely presented special inverse monoid.

Conversely, suppose $M = \mathrm{Inv}\langle A \mid R \rangle$ is a finitely (or recursively) presented special inverse monoid, and let $e \in M$ be an idempotent represented by some word $w$ over $A^{\pm 1}$. Since one can recursively enumerate relations which hold in $M$, one can enumerate all words in the $\GreenH$-class of $e$ (by listing a word $v$ as soon as one discovers that the relations $vv'=w$ and $v'v = w$ hold in $M$), and all relations which hold between such
words. Thus, the $\GreenH$-class of $e$ is a recursively presented group.
\end{proof}

The fact that the construction in Theorem~\ref{thm_stabiliser} preserves the number of relators in the underlying group (except when free) means it can
be applied in particular to one-relator inverse monoids.

\begin{corollary}\label{cor_onerelator}
Every finitely generated subgroup of a one-relator group arises as a group $\GreenH$-class in a one-relator special inverse monoid.
\end{corollary}

We do not know any examples of one-relator special inverse monoids ($E$-unitary or otherwise) containing group $\GreenH$-classes which are not embeddable in a one-relator group, or which are not finitely generated.
\begin{question}
Is every group $\GreenH$-class in a one-relator special inverse monoid necessarily embeddable in a one-relator group?
\end{question}
\begin{question}\label{question_onerelatorfp}
Is every group $\GreenH$-class in a one-relator special inverse monoid necessarily finitely generated?
\end{question}
Recall that a group is said to have the \textit{Howson property} (or \textit{finitely generated intersection property}) if the intersection of two finitely generated subgroups
is always finitely generated. Free groups have this property \cite[Proposition~3.13]{Lyndon:2001lh} but in contrast there are one-relator groups (even hyperbolic ones) which do not \cite{KarrassSolitar1969,Kapovich1999}. However, we do not know of any example of a one-relator group containing two \textit{conjugate} finitely generated subgroups whose intersection is not finitely generated; nor are we aware of any proof that this cannot happen. If it can happen then this would imply a negative answer to Question~\ref{question_onerelatorfp}. Indeed, by applying Lemma~\ref{lemma_stabiliser} and Theorem~\ref{thm_stabiliser} we could obtain an $E$-unitary one-relator special inverse monoid with 
a group $\GreenH$-class that is a finite
index overgroup of a group that is not finitely generated, and hence by \cite[Proposition~4.2]{Lyndon:2001lh} not itself not finitely generated.
\begin{question}
Is the intersection of two conjugate finitely generated subgroups in a one-relator group necessarily finitely generated?
\end{question}

Another application of Theorem~\ref{thm_stabiliser} is to construct examples of finitely presented special inverse monoids with infinitely many pairwise non-isomorphic group $\GreenH$-classes. This contrasts sharply with the case of special (non-inverse) monoids, where by a result of Malheiro \cite{Malheiro2005} all idempotents lie in the $\GreenD$-class of $1$, and therefore all group $\GreenH$-classes are necessarily isomorphic to each other.
\begin{corollary}\label{cor_everyfinitegroup}
There exists an $E$-unitary finitely presented special inverse monoid (in fact, a free product of a finitely presented group with a finite rank free inverse monoid) in which every finite group arises as a group $\GreenH$-class.
\end{corollary}
\begin{proof}
Take $G$ to be any finitely presented group with every finite group as a subgroup (for example, Higman's universal group into which embeds every finitely generated group \cite[Theorem~7.3]{Lyndon:2001lh}), and $H$ to be the trivial group. Then every finite group arises as a finite union of left $H$-cosets and is its own stabiliser under left translation, so Theorem~\ref{thm_stabiliser}
gives a finitely presented $E$-unitary special inverse monoid (in fact, a free product of $G$ with a finite rank free inverse monoid) in which they all arise as group
$\GreenH$-classes.
\end{proof}

We conclude by remarking on a natural and related, if slightly tangential, open question.
Belyaev \cite{Belyaev} showed that an analogue of Higman's embedding theorem holds for inverse semigroups: every recursively presented inverse
semigroup embeds in a finitely presented one. We do not know if a corresponding result holds for special inverse monoids.
\begin{question}
Does every recursively presented special inverse monoid embed as a subsemigroup (or even as a submonoid) in a finitely presented special inverse monoid?
\end{question}

The results of this section also lead naturally to the question of what can be said about maximal submonoids of a finitely presented special inverse monoid $M$. These are the submonoids of the form $eMe$ where $e$ is an idempotent of $M$. Can these monoids be completely described? The results of this section already give some information about this class of monoids: for example the group of units of $eMe$ is equal to the $\gh$-class $H_e$ of $M$ and hence need not be finitely generated.

\end{document}